\documentclass[12pt]{article}

\usepackage[utf8]{inputenc}
\usepackage[T1]{fontenc}
\usepackage{lmodern}
\usepackage[english]{babel}

\usepackage[top=2.5cm, bottom=2.5cm, left=2.5cm, right=2.5cm]{geometry}
\usepackage{setspace}
\usepackage{amsmath, amssymb, amsthm, mathtools, bm, mathrsfs}

\usepackage{graphicx, subcaption, float, tikz}
\usetikzlibrary{arrows.meta, positioning, shapes.geometric}

\usepackage{array, booktabs, tabularx, multirow, adjustbox}

\usepackage{algorithm, algorithmic}

\usepackage{enumitem}

\usepackage{titlesec}

\usepackage[authoryear]{natbib} 
\usepackage{hyperref}
\hypersetup{
    colorlinks=true,
    linkcolor=blue,
    citecolor=red,
    urlcolor=cyan
}

\newtheorem{thm}{Theorem}
\newtheorem{lemma}{Lemma}

\titleformat{\section}{\normalfont\fontsize{14}{16}\bfseries}{\thesection}{1em}{}
\titleformat{\subsection}{\normalfont\fontsize{12}{14}\bfseries}{\thesubsection}{1em}{}
\titlespacing*{\section}{0pt}{12pt plus 4pt minus 2pt}{6pt plus 2pt minus 2pt}
\titlespacing*{\subsection}{0pt}{8pt plus 2pt minus 2pt}{4pt plus 2pt minus 2pt}

\begin{document}

\begin{center}
{\rmfamily\fontsize{18}{22}\selectfont\textbf{Weak Mixing Property for\\ Linear Involutions}}
\end{center}

\begin{center}
\rmfamily\textbf{Erick Gordillo}

Universität Heidelberg\\
egordillo@mathi.uni-heidelberg.de
\end{center}

\begin{center}
\rmfamily\textbf{\textit{Abstract}}
\end{center}

\doublespacing
\noindent
In this work, we extend the celebrated result of Avila--Forni~\cite{avila2007weak}  on the weak mixing property of interval exchange transformations to the setting of linear involutions, which naturally arise from the study of vertical foliations on half-translation surfaces. Using recent advances on the Kontsevich--Zorich cocycle for quadratic differentials~\cite{belldiagonal, gutierrez2019classification, trevino2013non}, we establish that, for every dynamically irreducible generalized permutation, the associated linear involution is weakly mixing for almost every admissible parameter.
\medskip
\noindent


\singlespacing
\section*{Introduction}
Interval exchange transformations (IETs) are a class of dynamical systems that naturally arise from the study of translation flows on translation surfaces. Given a translation surface and a chosen direction, one can take a transversal segment to the corresponding translation flow and consider the first return map. This map divides the segment into finitely many subintervals (say $d$), and acts on each as a translation. An IET is completely determined by a permutation $\pi$, describing the combinatorial rearrangement of the intervals, and a vector $\lambda\in \mathbb{R}^d_+$, whose coordinates give the lengths of the subintervals.

Over the past  decades, IETs have become a central object of study in ergodic theory and Teichmüller dynamics, one of the most relevant results deals with spectral properties of this objects, namely

\begin{thm} \cite{avila2007weak}
    For an irreducible permutation $\pi$ on an alphabet of $d$ letters  which is not a rotation, and for almost any parameter $\lambda\in \mathbb{R}_+^d$, the interval exchange transformation $T(\pi,\lambda)$ is weakly mixing.
\end{thm}

This result comes as a consequence of a technical and strong machinery for the theory of measurable cocycles that Avila and Forni developed in \cite{avila2007weak} and the fact that the Kontsevich-Zorich cocycle is non-uniformly hyperbolic almost everywhere, proved by  Forni in 2002 \cite{forni2002deviation} and taken to a major generalization in 2007 by Avila and Viana proving that almost everywhere, that cocycle has simple Lyapunov spectrum \cite{avila2007simplicity}. 

Weak mixing for a measure-preserving transformation \( T \) on a probability space can be characterized in two equivalent ways: through Cesàro averages or through the spectral properties of its associated unitary operator.  

From the first viewpoint, \( T \) is weakly mixing if, for every pair of square-integrable functions \( f, g \in L^2(\mu) \), the averaged correlations between \( f \circ T^n \) and \( g \) converge to the product of their means:  
\[
\frac{1}{N} \sum_{n=0}^{N-1} \left| \int (f \circ T^n) g \, d\mu - \int f \, d\mu \int g \, d\mu \right| \longrightarrow 0 \quad \text{as } N \to \infty.
\]
Intuitively, this means that the dynamics progressively ``mix'' observables, eliminating statistical dependencies over time.  

From the spectral viewpoint, weak mixing is equivalent to the absence of non-trivial eigenfunctions for the unitary operator \( U_T \colon L^2(\mu) \to L^2(\mu) \) defined by \( U_T f = f \circ T \). In this case, the spectrum of \( U_T \) is purely continuous (no eigenvalues of modulus \( 1 \)), and the system is said to have a \emph{simple spectrum}.

Translation surfaces can be described as pairs \((S, \omega)\), where \(S\) is a Riemann surface and \(\omega\) is an abelian differential. This framework extends naturally to quadratic differentials: a translation surface corresponds to a pair \((S, q)\) where \(q\) is a quadratic differential that is globally the square of an abelian differential. The vertical foliation determined by \(q\) coincides with the vertical flow on the surface. This setting accounts for a full-measure subset of all quadratic differentials.  

When \(q\) is not globally a square, the pair \((S, q)\) defines a \emph{half-translation surface}, whose vertical foliation is \emph{non-orientable}. In this case, it is natural to ask whether the weak mixing property, well established for one-dimensional dynamics arising from orientable cases, persists in this measure-zero subset of quadratic differentials. The present work addresses this question.  

A key tool here is the notion of \emph{linear involutions}, which generalize interval exchange transformations (IETs) to encode the dynamics of vertical foliations on half-translation surfaces through one-dimensional maps. Unlike IETs, however, linear involutions cannot be obtained as first return maps to a transversal of a flow.  

In Appendix~A of \cite{solomyak2024note}, ~Hubert and ~Matheus present conditions on \(S\)-adic systems under which one can adapt the Avila--Forni approach to prove generic weak mixing. These ideas can be applied  to linear involutions. The aim of this paper is to make this connection precise, relating the Avila--Forni philosophy to the properties of the Kontsevich--Zorich cocycle for quadratic differentials. Our main result is as follows:

\begin{thm}\label{theorem2}
    Let $\pi$ be a dynamically irreducible generalized permutation, and let $\lambda$ be an admissible parameter such that the half-translation surface $S$ generated by $( \pi,\lambda)$ has genus $g(S) > 1$. Then, for almost every such $\lambda$, the linear involution $T(\pi,\lambda)$ is weakly mixing.
\end{thm}

\textbf{Remark.} If we restrict to parameters \(\lambda\) satisfying \(|\lambda| = 1\), then “almost every” in Theorem~\ref{theorem2} refers to Lebesgue measure of dimension \(d-2\), where \(d\) is the number of intervals in the linear involution.

The strategy of Avila--Forni can be outlined as follows. For an interval exchange transformation \(T(\lambda, \pi)\) with \(d\) subintervals, suppose there exists a measurable function \(f : I \to \mathbb{C}\) such that  
\begin{equation} \label{eq1}  
    f(T(x)) = e^{2\pi i v_j} f(x), \quad x \in I_j,  
\end{equation}  
for some vector \(v\) in the space where the Lyapunov exponents of the Rauzy--Veech cocycle are non-zero. The \emph{Veech criterion} \cite{veech1984metric} asserts that such a vector \(v\) must lie in the weak stable space of the Rauzy--Veech renormalization scheme. Avila--Forni show, via an elimination process, that this \emph{undesirable} situation occurs only on a null set: for almost every parameter \(( \pi,\lambda)\), any vector \(v\) satisfying \eqref{eq1} must in fact belong to the integer lattice \(\mathbb{Z}^d\). Membership in \(\mathbb{Z}^d\) forces all eigenfunctions to correspond to a purely continuous spectrum, which implies weak mixing.

To carry out this elimination, Avila--Forni exploit the fact that the Rauzy--Veech cocycle is non-uniformly hyperbolic \cite{forni2002deviation}, which was improved by \cite{avila2007simplicity} with the Zariski density of the monoid generated by the Rauzy--Veech matrices. They then prove a general theorem for integral, locally constant, uniform cocycles, a class which includes the Rauzy--Veech cocycle, ensuring that the exceptional set where \eqref{eq1} admits a non-trivial solution has measure zero.

 In Lemma~\ref{veechcriteria}, we present a precise formulation of a Veech criterion for linear involutions, adapting the ideas from Veech’s original proof in \cite{veech1984metric}. The philosophy behind the Veech criterion for weak mixing is quite general, applying to a wide range of dynamical systems (see \cite{mercat2024coboundaries}). In our setting, the proof relies on the fact that linear involutions faithfully encode the dynamics of vertical foliations on half-translation surfaces. As shown in \cite{boissy2009dynamics}, such surfaces can be decomposed into finitely many rectangles, and the vector \(v\) appearing in Equation~\eqref{eq1} corresponds to the geometric parameters of this decomposition.  

To apply the Avila--Forni strategy in the context of linear involutions, we focus on the ergodic properties of the Rauzy--Veech cocycle in this setting. Starting with a half-translation surface \((S, q, \Sigma)\) arising from the "suspension" of a linear involution, we pass to its double orientation cover \(\pi : R \to S\), where \((R, \pi^* q, \hat{\Sigma})\) is a translation surface. The homology \(H_1(R, \mathbb{R})\) admits an invariant splitting into \emph{plus} and \emph{minus} subspaces, \(H^+(R)\) and \(H^-(R)\), under the Kontsevich--Zorich cocycle on \(R\). The restriction of the cocycle to the plus part corresponds to its action on the homology of the translation surface \(S\), while the minus part encodes the evolution of the rectangle-decomposition parameters of \(S\) under the Rauzy--Veech renormalization (see \cite{gutierrez2019classification}). This minus component can thus be interpreted as the renormalization scheme for linear involutions described by Boissy--Lanneau~\cite{boissy2009dynamics}.

In \cite{trevino2013non}, ~Treviño proved that, for every connected component of a stratum of quadratic differentials \(\mathcal{C}\), the second Lyapunov exponent of both the plus and minus parts of the Kontsevich--Zorich cocycle is positive for almost every \(q \in \mathcal{C}\). While this positivity result is significant, in our setting even if non-uniform hyperbolicity alone seems sufficient to exclude problematic parameters, due to the more intricate combinatorics in the non-orientable case, as explained in \cite{boissy2009dynamics} it is not straightforward to adapt the arguments of Section~5 in \cite{avila2007weak} directly.  

This difficulty can be overcome by appealing to results from \cite{belldiagonal}, where the authors show that, for every connected component of a stratum \(\mathcal{C}\), the monoids generated by the plus and minus parts of the Kontsevich--Zorich cocycle are Zariski dense in their respective symplectic ambient spaces. Zariski density implies that the cocycle is both twisting and pinching, which, in particular, allows to deduce the existence of the matrices manually constructed  by Avila and Forni to exclude parameters in Section 5 \cite{avila2007weak}. Combining this with the Avila--Viana criterion \cite{avila2007simplicity} yields simplicity of the Lyapunov spectrum. Moreover, as explained in Section~7 of \cite{belldiagonal}, the same conclusion holds for the Lyapunov spectrum of the discrete version of the Kontsevich--Zorich cocycle. These results allow us to eliminate the \emph{undesirable} dynamics in the case of linear involutions.  

It is natural to attempt to follow \cite{avila2007weak} and establish an analogous weak mixing result for vertical foliations on half-translation surfaces. In the non-orientable case, however, the vertical foliation cannot be defined as a flow. To study weak mixing for minimal foliations in this setting, one must instead pass to the double orientation cover, obtaining a translation flow whose first return map to a transversal is an \emph{interval exchange transformation with involution} in the sense of \cite{avila2012exponential}. Then  one would need to prove a Veech criterion for such systems and deduce generic weak mixing for the flows. Nevertheless, in fact, a stronger statement has been obtained in \cite{arana2024weak}, where it is shown that the only cases in which weak mixing fails are square-tiled surfaces. This covers, in particular, the generic translation surfaces obtained from half-translation surfaces via the double orientation cover.  \\

\textbf{Acknowledgments}: I would like to thank Carlos Matheus for explaining the ideas of Avila--Forni and how to use Zariski density, also I would like to thank Rodolfo Gutierrez-Romo for explaining in detail the ideas of the decomposition of the Kontsevich-Zorich cocycle.

\section{Preliminaries}

\subsection{Half-translation surfaces}
Let $S$ be a connected Riemann surface of genus $g$, and let $\Sigma = {z_1, \ldots, z_s}$ be a finite subset associated with a pattern $\kappa = (n_1, \ldots, n_s)$, where each $n_i \in \{-1\} \cup \mathbb{N}$, satisfying $\sum_i n_i = 4g - 4$. A half-translation structure on $S$ is defined as an atlas on $S \setminus \Sigma$ such that for any two overlapping charts $\phi_i$ and $\phi_j$, the transition maps $\phi_i \circ \phi_j^{-1}$ takes the form $\phi_i(z) = \pm \phi_j(z) + c$ for some constant $c$. Additionally, around each point $z_i \in \Sigma$, there exists a neighborhood that is isometric to a Euclidean cone, for $S\setminus{\Sigma}$ there exists a flat metric, for which the points in $\Sigma$ are singularities of order $n_i$. Observe that with this construction, the holonomy group (derived from considering parallel transport along loops in the surface) is a subgroup of $\mathbb{Z}_2$.\\

Another way to present this structure is through quadratic differentials. Formally, a meromorphic quadratic differential is defined as a non-zero section of the symmetric square of the cotangent bundle of the surface $S$, specifically $q \in \mathcal{S}^2 T^*S$. In local coordinates $z$, it can be expressed as $q = f(z)  dz \otimes dz$ (or more commonly as $q = f(z)dz^2$), where $f$ is a meromorphic function. If we can define a quadratic differential $q = f(z) dz^2$ on a surface $S$ with singularities located in a set $\Sigma$, characterized by orders described by $\kappa$, then we say that $(S, q, \Sigma)$ is a \emph{half-translation surface}.\\

When the quadratic differential can be expressed as the \emph{square} of an abelian differential $\omega$, the resulting structure is referred to as a \emph{translation surface}. In this case, the atlas described above has the property that its transition maps are given by $\phi_i(z) = \phi_j(z) + c$, which implies that the holonomy group is trivial.\\

In a half-translation surface $(S, q, \Sigma)$, we can define two transversal foliations on $S \setminus \Sigma$ called the vertical $\mathcal{F}^v_q$ and horizontal $\mathcal{F}^h_q$ foliations. These foliations are given by the integral curves of $\ker(\text{Re}(q^{1/2}))$ and $\ker(\text{Im}(q^{1/2}))$, respectively. Both foliations are measurable, and the points in $\Sigma$ represent $n_i$-pronged singularities for the foliation. One can verify that the foliation is globally orientable if and only if every $n_i$ is even. This is the case when the quadratic differential is the global square of an abelian differential. (See \cite{farb2011primer} for a careful treatment). For a half-translation surface, we will say that a singularity is \emph{non-orientable} if the number of prongs is odd.\\

There is a natural distinction between quadratic differentials that arise as the global square of an abelian differential and those that do not. The former are called \emph{orientable}, while the latter are termed \emph{non-orientable}. In this work, we focus on the non-orientable class.\\

Consider a Riemann surface $S$ of genus $g \geq 1$. The Teichmüller space of its meromorphic quadratic differentials, denoted $\mathcal{T} \mathcal{L}_g$, is defined as the quotient of the set of meromorphic quadratic differentials by the group of orientation-preserving diffeomorphisms on $S$ that are isotopic to the identity. Specifically, $q \sim q'$ if there exists $f \in \mathrm{Diff}_0^+(S)$ such that $f^*(q) = q'$. We can consider $Q_g$ as the subspace that arise from considering non-orientable quadratic differentials in $\mathcal{T} \mathcal{L}_g$. \\

Furthermore, we define the moduli space $\mathcal{M}_g$ as the quotient of $\mathcal{T} \mathcal{L}_g$ by the mapping class group $\Gamma_g$ of the surface. We can partition $\mathcal{M}_g$ into subsets corresponding to quadratic differentials that arise as the square of an abelian differential and those that do not. The latter are central to our discussion, and we refer to the moduli space of non-orientable quadratic differentials as $\mathcal{Q}_g$.\\

These moduli spaces are naturally stratified by the orders of the singularities. If $\kappa$ denotes a singularity pattern for elements in $(S, q, \Sigma)$, given by $\kappa = (n_1, \ldots, n_s)$ with \mbox{$\sum_i n_i = 4g - 4$}, we can define
$$\mathcal{Q}_\kappa=\mathcal{Q}_g\cap \{ q\:\: with \:\: pattern\:\: of \:\: singularities \:\: \kappa\}.$$

That definition of course works for considering the strata of non-orientable quadratic differentials with that pattern of singularities before quotienting by $\Gamma_g$, defined as $Q_\kappa$. These stratified spaces are complex orbifolds of dimension $2g + s - 2$, and they are not necessarily connected. E. Lanneau provided a complete classification of the connected components of these spaces \cite{lanneau2008connected}. It is worth mentioning that if $q$ is an orientable quadratic differential defined on a surface of genus $g$, then its pattern of singularities $\kappa=(n_1,...,n_s)$ fulfills that  $\sum_i n_i=2g-2.$\\

Consider a natural action on $\mathcal{Q}_\kappa$ by $SL(2, \mathbb{R})$ described as follows: if $(S, q, \Sigma)$ is a half-translation surface, we start with the maximal atlas on $S \setminus \Sigma$. For any $A \in SL(2, \mathbb{R})$, define $A \cdot (S, q, \Sigma)$ as the half-translation structure obtained by post-composing the original atlas of $(S, q, \Sigma)$ with $A$. This action preserves the orders of the singularities in $\Sigma$, and therefore it preserves the stratification of the moduli spaces.\\

There exists a \textit{natural} invariant measure on $\mathcal{Q}\kappa$ in the Lebesgue class, denoted by $\lambda\kappa$, which, however, has infinite mass. To address this, we proceed as follows: a quadratic differential $q$ on a Riemann surface $S$ defines an area form given by $A(q)=\int_S |q|$.\\

We restrict our attention to the set of quadratic differentials that yield area $1$. By considering the appropriate quotient, we define the hypersurface $\mathcal{Q}\kappa^{(1)}$. By disintegrating the measure $\lambda\kappa$ along the level sets of the area function $q \mapsto A(q)$, we obtain a Lebesgue measure $\lambda_\kappa^{(1)}$ on $\mathcal{Q}_\kappa^{(1)}$.\\

A key action on $\mathcal{Q}_\kappa^{(1)}$ is the one induced by the diagonal group, represented by the matrices:
$$\left \{g_t= \begin{pmatrix}
e^t & 0 \\
0 & e^{-t}
\end{pmatrix} \:\: t\in \mathbb{R} \right \} \leq SL(2,\mathbb{R}).$$

The action of this group, known as the \emph{diagonal flow}, plays a crucial role in the renormalization process, as we will see shortly. Furthermore,  H. Masur proved for the principal stratum ($\kappa=(1,...,1))$ \cite{masur1982interval} and then W. Veech for any stratum \cite{veech1986teichmuller}  that the diagonal flow acts ergodically on each $\mathcal{Q}\kappa^{(1)}$ with respect to the  finite measure $\lambda\kappa^{(1)}$. As a consequence, it can be shown that for almost every quadratic differential $q$ on $S$, its vertical foliation is uniquely ergodic. This result was later strengthened by Kerkhoff, Masur, and Smillie, who demonstrated that for every quadratic differential $q$ on $S$, the foliation in almost any direction $\theta \in [0, 2\pi)$ is uniquely ergodic \cite{kerckhoff1986ergodicity}.\\

Consider a construction called the \emph{double orientation cover}, which takes a non-orientable quadratic differential \( q \) on \( S \) and produces an orientable quadratic differential. This construction can be thought of as the unique (up to isomorphism) double-sheeted cover of \( S \), branched over the singularities of odd order, such that it \textit{unwinds} the non-orientable parts of \( S \).\\

Let \( (S, q, \Sigma) \) be a non-orientable quadratic differential, where \( \{ (z_j, U_j) \} \) is an atlas on \( S \setminus \Sigma \). We assume that \( q \) is locally defined as \( f_j(z_j) \, dz_j^2 \) on each \( U_j \). For each set \( U_j \), consider two copies, \( U_j^+ \) and \( U_j^- \). On each \( U_j \), let \( g_j^+ \) and \( g_j^- \) be the two square roots of \( f_j(z_j) \).\\

The first step of the construction is to consider the natural maps \( \pi_j^\pm: U_j^\pm \to U_j \). For the gluings, we first consider the case of overlapping charts \( U_j \cap U_i \). If $$ g_j^+(z_j(p)) \frac{dz_j}{dz_i}(p) = g_i^+(z_i(p))$$ (the orientable case), we identify \( (\pi_j^\pm)^{-1}(U_i \cap U_j) \) with \( (\pi_i^\pm)^{-1}(U_i \cap U_j) \). In the case when $$ g_j^+(z_j(p)) \frac{dz_j}{dz_i}(p) = g_i^-(z_i(p)) $$ (the non-orientable part), we identify \( (\pi_j^\mp)^{-1}(U_i \cap U_j) \) with \( (\pi_i^\pm)^{-1}(U_i \cap U_j) \). These identifications define a non-branched cover over \( S \setminus \Sigma \).\\

Moreover, the expressions \( g_j^\pm \circ z_j \circ \pi_j^\pm \, d(z_j \circ \pi_j^\pm) \) locally define an abelian differential \( \omega \). Consequently, if \( R \) is the underlying translation surface with the branched cover \( \pi: R \to S \), we have that \( \pi_* q = \omega^2 \).\\

If \( z_i \) is a pole on \( S \), then \( \pi^{-1}(z_i) \) becomes a marked regular point. If \( z_i \) is a singularity with odd order \( n_i \), then in the translation surface, it will have order \( n_i + 1 \). Finally, if \( z_i \) is a singularity of even order, \( \pi^{-1}(z_i) \) will split into two singularities of order \( \frac{n_i}{2} \).\\

Thus, if \( \kappa = (n_1, \ldots, n_s, n_{s+1}, \ldots, n_{s+r}, p_1, \ldots, p_t) \) is the pattern of singularities of a quadratic differential \( (S, q, \Sigma) \), where \( n_1, \ldots, n_s \) have odd order and \( p_1, \ldots, p_t \) are poles, then the translation surface \( (R, \omega, \hat{\Sigma}) \) satisfies \( \hat{\Sigma} = \pi^{-1}(\Sigma) \setminus \pi^{-1}(\{ p_1, \ldots, p_t \}) \) and has a pattern of singularities given by \( \hat{\kappa} = (n_1+1, \ldots, n_s+1, \frac{n_{s+1}}{2}, \frac{n_{s+1}}{2}, \ldots, \frac{n_{s+r}}{2}, \frac{n_{s+r}}{2}) \). This construction yields an embedding \( \mathcal{Q}_\kappa \hookrightarrow \mathcal{H}_{\hat{\kappa}} \).\\

In the double orientation cover \( R \), there exists a canonical involution \( \iota \) that interchanges \( U_j^\pm \) with \( U_j^\mp \). It is important to note that, in this context, the double orientation cover\\ $\pi~: R~\to~S$ can be viewed as the quotient map \( \pi: R \to R / \iota \). The involution \( \iota \) induces an action on the homology \( H_1(R, \mathbb{R}) \), which gives rise to a decomposition of homology into invariant and anti-invariant parts:

\[
H_1^+(R, \mathbb{R}) = \{ c \in H_1(R, \mathbb{R}) \mid \iota_*(c) = c \},
\]
\[
H_1^-(R, \mathbb{R}) = \{ c \in H_1(R, \mathbb{R}) \mid \iota_*(c) = -c \}.
\]

This decomposition separates the invariant part \( H_1^+(R, \mathbb{R}) \) and the anti-invariant part \( H_1^-(R, \mathbb{R}) \). Moreover, \( H_1^+(R, \mathbb{R}) \) is isomorphic to \( H_1(S, \mathbb{R}) \), as shown by the following argument:\\

The induced transformations on homology satisfy the relation \( \pi_* \circ \iota_* = \pi_* \). Therefore, if \( c \in H_1^-(R, \mathbb{R}) \), then \( \pi_*(c) = -\pi_*(c) = 0 \), which implies that \( H_1^-(R, \mathbb{R}) \subseteq \ker \pi_* \).\\ 

On the other hand, any curve \( \gamma \) in \( S \) can be lifted to two curves \( \gamma_1 \) and \( \gamma_2 \) in \( R \). Note that \( [\gamma_1] + [\gamma_2] \in H_1^+(R, \mathbb{R}) \), and \( \pi_*([\gamma_1] + [\gamma_2]) = 2[\gamma] \). By linearity, we conclude that the restriction \( \pi_* |_{H_1^+(R, \mathbb{R})}: H_1^+(R, \mathbb{R}) \to H_1(S, \mathbb{R}) \) is an isomorphism.
This isomorphism plays a crucial role in the understanding of the Kontsevich-Zorich cocycle on non-orientable quadratic differentials.

\subsection{Linear Involutions}
As mentioned earlier, when analyzing the vertical foliation of a half-translation surface that is not globally the square of an abelian differential, the classical first return map to a transversal section in the northward direction is insufficient. In such cases, the foliation's non-orientability produces subintervals where the first return map is periodic of period~2, failing to capture the true dynamics. To address this, Danthony and Nogueira \cite{danthony1990measured} introduced the notion of a \emph{linear involution}, described as follows.

Consider two copies of an interval \(X\), denoted \(\overline{X} = X \times \{0,1\}\). For a point \(x \in X \times \{1\}\), follow its trajectory along the vertical foliation northwards. If the first return to \(\overline{X}\) occurs at \((x_1,1)\), switch to \((x_1,0)\) and continue southwards. Conversely, if the return occurs at \((x_1,0)\), switch to \((x_1,1)\) and continue northwards. For points in \(X \times \{0\}\), the procedure starts southwards, with analogous switching upon return. This extended process accurately models the vertical foliation.

Singularities in the half-translation surface imply that some points in \(\overline{X}\) are not defined under this process, naturally partitioning \(\overline{X}\) into subintervals. The half-translation structure ensures that the dynamics on each subinterval is an isometry: orientation-reversing if the point returns to the same component, and orientation-preserving if it returns to the opposite component. Consequently, the number of subintervals is always even.

To formalize this as a dynamical system, one encodes the combinatorics of the subintervals via a \emph{generalized permutation} recording their order.  

Let \(\mathcal{A}\) be an alphabet of \(d\) letters, and let \(\pi: \{1,\dots,2d\} \to \mathcal{A}\) be a two-to-one map, represented as
\[
\pi = 
\begin{pmatrix}
\pi(1) & \cdots & \pi(l) \\
\pi(l+1) & \cdots & \pi(l+m)
\end{pmatrix},
\]
where \(l + m = 2d\). Note that \(l\) and \(m\) are not necessarily equal. If \(\pi\) is equipped with a fixed-point-free involution \(\sigma\) such that \(\pi^{-1}(i) = \{i, \sigma(i)\}\), we call \(\pi\) a \emph{generalized permutation}.  

Consider vectors \(\lambda \in \mathbb{R}_+^d\) satisfying
\begin{equation} \label{parameterequation}
    \sum_{i \leq l} \lambda_{\pi(i)} = \sum_{i > l} \lambda_{\pi(i)}.
\end{equation}
Assuming this set is non-empty, we construct a linear involution as follows.

Take two copies of an interval of length equal to the common sum in \eqref{parameterequation}, partitioning the first into \(l\) subintervals and the second into \(m\) subintervals according to \(\lambda_{\pi(i)}\). The generalized permutation \(\pi\) prescribes the order, while \(\lambda\) determines the lengths. The map is defined by:
\begin{itemize}
    \item Identifying subintervals with the same label in the same component via an orientation-reversing isometry.  
    \item Identifying subintervals in different components via an orientation-preserving isometry.
\end{itemize}

The transformation proceeds in two steps:
\begin{enumerate}
    \item Match subintervals with the same label according to \(\mathcal{A}\), using the appropriate isometry.  
    \item Apply an involution exchanging the two components.
\end{enumerate}

The pair \((\pi, \lambda)\) thus defines a linear involution (see Figure~\ref{linearinvolution:fig}), which accurately models the vertical foliation dynamics on a half-translation surface.

Observe that, if the set of \(\lambda\) satisfying \eqref{parameterequation} is non-empty, it forms the intersection of the positive cone \(\mathbb{R}^d_+\) with a hyperplane, yielding a \((d-1)\)-dimensional set, which we denote by \(\overline{W}_\pi\).

Observe that if for a generalized permutation  $\pi$, the set of $\lambda$ satisfying \ref{parameterequation} is not empty. Then it is the intersection of the positive cone $\mathbb{R}^d$ with a hyperplane, thus it has dimension $d-1$. 

\begin{figure} 
    \centering

\begin{tikzpicture}[scale=2] 
\coordinate (A) at (0, 0);        
    \coordinate (B) at (2, 1);      
    \coordinate (C) at (3, .5);        
    \coordinate (D) at (3, -0.5);     
    \coordinate (E) at (1, -1.5);       
    \coordinate (F) at (0, -1);       
        
    \coordinate (B1) at (5, 1.5);      
    \coordinate (C1) at (6, 1);        
    \coordinate (D1) at (6, 0);     
    \coordinate (E1) at (4, -1);

    \draw[thick] (A) -- (B) -- (C)-- (B1)-- (C1) -- (D1) -- (E1)-- (D)-- (E) -- (F) -- cycle;
 \coordinate (M1) at (1, 0.5);      
    \coordinate (M2) at (2.5, 0.75);   
    \coordinate (M3) at (4, 1.1);   
    \coordinate (M4) at (5.5, 1.25);   
    \coordinate (M5) at (6, .5);    
    \coordinate (M6) at (5, -0.5);     
    \coordinate (M7) at (3.4, -.75);  
    \coordinate (M8) at (2.2, -1.05);  
    \coordinate (M9) at (0.5, -1.2);   
    \coordinate (M10) at (0, -0.5);    
    \node at (M1) [above left] {I};
    \node at (M2) [above right] {II};
    \node at (M3) [above right] {I};
    \node at (M4) [above right] {III};
    \node at (M5) [right] {IV};
    \node at (M6) [below right] {V};
    \node at (M7) [below left] {II};
    \node at (M8) [below left] {V};
    \node at (M9) [below left] {III};
    \node at (M10) [above left] {IV};

    \draw[thick] (.5,.2)--(5.9,.2);
    \draw[thick] (.5,0)--(5.9,0);
    \draw[dotted, thick, red, ->] (1.1,.2)--(1.1,.58);
     \draw[dotted, thick, red, ->] (4.3,1.1)--(4.3,.2);
     \draw[dotted, thick, blue, ->] (3.6,0)--(3.6,-.8);
     \draw[dotted, thick, blue, ->] (2.5,.7)--(2.5,.2);
     \draw[thick] (2,.17)--(2,.23);
     \draw[thick] (3,.17)--(3,.23);
     \draw[thick] (4.5,.17)--(4.5,.23);
     \draw[thick] (5,.17)--(5,.23);
     \draw[thick] (5.5,.17)--(5.5,.23);
     \draw[thick] (1,-.03)--(1,.03);
     \draw[thick] (2.9,-.03)--(2.9,.03);
     \draw[thick] (3.9,-.03)--(3.9,.03);
     
\draw[thick] (.5,-3.3)--(5.9,-3.3);
\draw[thick] (.5,-3.5)--(5.9,-3.5);

\draw[thick] (2,-3.53)--(2,-3.47);
\draw[thick] (3,-3.53)--(3,-3.47);
\draw[thick] (4.5,-3.53)--(4.5,-3.47);
\draw[thick] (5,-3.53)--(5,-3.47);
\draw[thick] (5.5,-3.53)--(5.5,-3.47);
\draw[thick] (1,-3.33)--(1,-3.27);
\draw[thick] (2.9,-3.33)--(2.9,-3.27);
\draw[thick] (3.9,-3.33)--(3.9,-3.27);

\node [scale=.8] at (1.5,-3.65) {$T(A_2)$};
\node [scale=.8] at (2.5,-3.65) {$T(B_2)$};
\node [scale=.8] at (3.8,-3.65) {$T(A_1)$};
\node[scale=.7] at (4.8,-3.65) {$T(C_2)$};
\node[scale=.7] at (5.3,-3.65) {$T(D_2)$};
\node[scale=.7] at (5.7,-3.65) {$T(C_1)$};
\node [scale=.8] at (.7,-3.2) {$T(D_1)$};
\node [scale=.8] at (1.8,-3.2) {$T(E_2)$};
\node [scale=.8] at (3.4,-3.2) {$T(B_1)$};
\node [scale=.8] at (5, -3.2)  {$T(E_1)$};

\draw[thick] (.5,-2.3)--(5.9,-2.3);
\draw[thick] (.5,-2.5)--(5.9,-2.5);

\draw[thick] (2,-2.33)--(2,-2.27);
\draw[thick] (3,-2.33)--(3,-2.27);
\draw[thick] (4.5,-2.33)--(4.5,-2.27);
\draw[thick] (5,-2.33)--(5,-2.27);
\draw[thick] (5.5,-2.33)--(5.5,-2.27);
\draw[thick] (1,-2.53)--(1,-2.47);
\draw[thick] (2.9,-2.53)--(2.9,-2.47);
\draw[thick] (3.9,-2.53)--(3.9,-2.47);
\node at (1.5,-2.2) {$A_1$};
\node at (2.5,-2.2) {$B_1$};
\node at (3.8,-2.2) {$A_2$};
\node at (4.8,-2.2) {$C_1$};
\node at (5.3,-2.2) {$D_1$};
\node at (5.7,-2.2) {$C_2$};
\node at (.7,-2.63) {$D_2$};
\node at (1.8,-2.63) {$E_1$};
\node at (3.4,-2.63) {$B_2$};
\node at (5, -2.63)  {$E_2$};
\end{tikzpicture}
  \caption{Note that the interval corresponding to $A_1$ has the same length as the interval corresponding to $A_2$ (same applies for the other letters). This enumeration is simply to clarify how the transformation $T$ operates. Additionally, recall that if the intervals corresponding to $\pi(i)$ and $\sigma(\pi(i))$ lie in the same connected component, the transformation acts as a translation followed by a flip.}
    \label{linearinvolution:fig}
\end{figure}
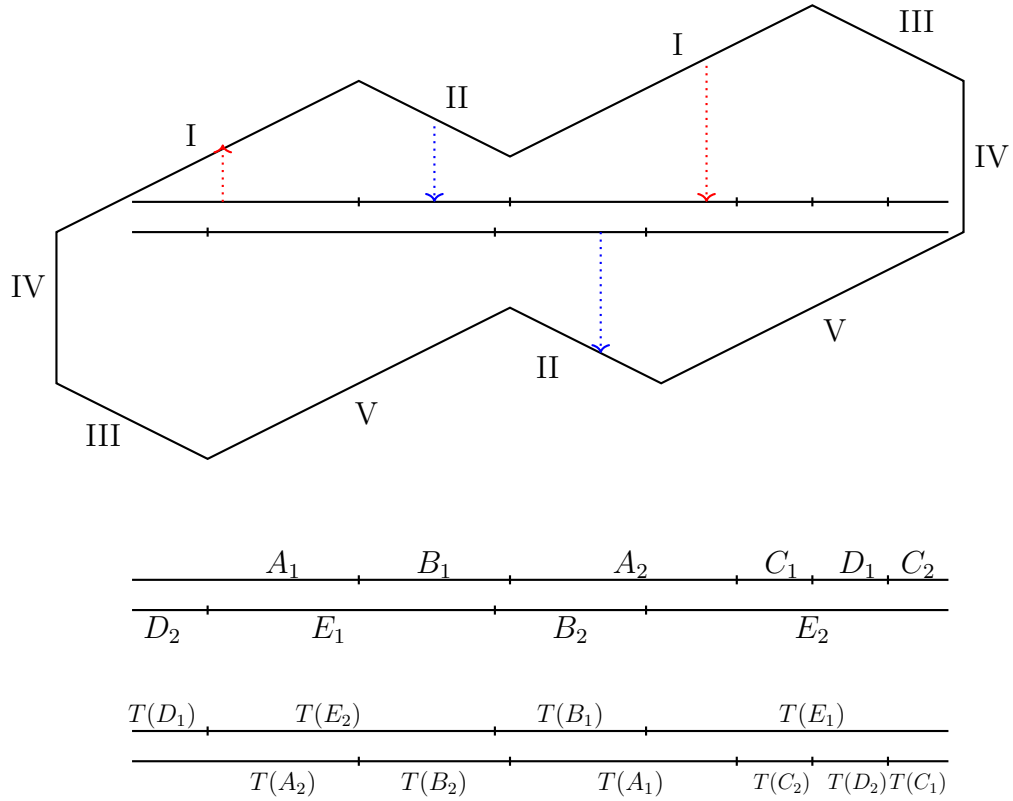

Since we are interested in linear involutions that arise from considering intersections with vertical foliations of half-translation surfaces, we will assume that for a linear involution $T(\pi,\lambda)$, the generalized permutation $\pi$ satisfies the following condition: there exist at least two numbers $i \leq l$ and $l+1 \leq j \leq l+m$ such that $\sigma(i) \leq l$ and $l+1 \leq \sigma(j) \leq l+m$.

\textbf{Notation:} We will denote the quantity:
\[
L = \sum_{\pi(i) \leq l} \lambda_{\pi(i)} = \sum_{i > \pi(l)} \lambda_{\pi(i)}.
\]

A key tool for studying dynamics in interval exchange transformations is a process known as \emph{Rauzy--Veech induction}. We can extend this process to linear involutions \( T(\pi, \lambda) \) by defining it as the first return map on the space 
\[
\overline{X}^{(1)} = \max\{L - \lambda_{\pi(l)}, L - \lambda_{\pi(l+m)}\} \times \{0,1\}.
\]
When this transformation is well-defined, the first return map is again a linear involution on the same number of subintervals. This process can be interpreted as a transformation \( \mathcal{Q} \) acting on the space of linear involutions on an alphabet \( \mathcal{A} \). Specifically, it induces a new generalized permutation $\pi^{(1)}$ and a new partition of subintervals for the space \( \overline{X}^{(1)} \). The transformed involution is denoted as \( \mathcal{Q}(\pi, \lambda) = (\pi^{(1)}, \lambda^{(1)}) \). 
\\

If the Rauzy--Veech induction is well-defined after \( n \) iterations, we denote the result as \( \mathcal{Q}^n(\pi, \lambda) = (\pi^{(n)}, \lambda^{(n)}) \), where for each $\pi^{(n)}$ we havet that the involution satisfies \( (\pi^{(n)})^{-1}(i) = \{i, \sigma^{(n)}(i)\} \). The space in which the transformation is defined is \(\overline {X^{(n)} }\), and the subintervals defining the linear involution are denoted as \( X_{\pi^{(n)}(i)}^{(n)} \). \\

A detailed description of the generalized permutations obtained after applying \mbox{Rauzy--Veech} induction can be found in \cite{boissy2009dynamics}. The transition from \( \lambda^{(1)} \) to \( \lambda \) can be described by a \( d \times d \) matrix. Namely $B(\pi,\lambda)\cdot\lambda^{(1)}=\lambda$.  For a linear involution denote by $\alpha_0$ the label for the $l$-th interval in the partition and $\alpha_1$ the $l+m$-th interval. If $\lambda_{\alpha_1}> \lambda_{\alpha_0}$ then we can define a matrix \[
B(\pi,\lambda)_{\alpha,\beta} =
\begin{cases}
1, & \text{if } \alpha=\beta, \text{or} \alpha=\alpha_1, \beta=\alpha_0 \\

0, & \text{In any other case}
\end{cases}
\]
We define an analogous matrix when $\lambda_{\alpha_0}>\lambda_{\alpha_1}$.
\\

Also if the Rauzy--Veech induction can be applied $n$ times to $(\pi,\lambda)$ then the matrix $$B^n(\pi,\lambda)=B(\pi^{(n-1)},\lambda^{(n-1)}) \cdot B(\pi^{(n-2)},\lambda^{(n-2)})\cdot ...\cdot B(\pi^{(1)},\lambda^{(1)})$$ satisfies that $$B^n(\pi,\lambda)\cdot \lambda^{(n)}=\lambda.$$

Observe that $\lambda_{\alpha_0} \neq \lambda_{\alpha_1}$ is a necessary condition for this process to be defined, but it is not necessarily sufficient. \\

The nomenclature of the renormalization will be regarded as \emph{top} or \emph{bottom} depending on which quantity, $\lambda_{\alpha_0}$ or $\lambda_{\alpha_1}$, is larger. There is an acceleration process for this renormalization, usually referred to as \textit{Zorich induction}, such that it only takes into consideration when it changes from top to bottom or bottom to top. The matrix is refered as $Z$.\\

Another interesting fact about this matrix, which we call Rauzy--Veech matrix is that it can be regarded as the \textit{visiting matrix}:
 A linear involution \( T(\pi, \lambda) \) defines a partition of two intervals into subintervals \( \{I_{\alpha}, I_{\sigma(\alpha)}\}_{\alpha \in \mathcal{A}} \). After applying one step of Rauzy--Veech induction, the induced linear involution defines a new partition \( \{I^{(1)}_\alpha, I^{(1)}_{\sigma(\alpha)}\}_{\alpha \in \mathcal{A}} \) of a proper subset of the initial components.  The matrix \( B(\pi, \lambda)^* \), the transpose of the Rauzy--Veech matrix, has an interpretation as a \emph{visiting matrix}.  Each entry \( B(\pi, \lambda)^*_{\alpha, \beta} \) counts the number of times one interval labeled by \( \alpha \) in the new partition visits one interval labeled by \( \beta \) in the previous partition before returning. One useful fact, that will be used in Lemma \ref{veechcriteria} is that this count is independent of the choice of the subinterval labeled by \( \alpha \). More generally, \( B^n(\pi, \lambda)^* \) is the visiting matrix  in the same sense from the partition associated to the linear involution \( (\pi^{(n)}, \lambda^{(n)}) \) to the original partition.\\

Consider a linear involution $T(\pi,\lambda)$ with $\{\lambda_\alpha\}_{\alpha \in \mathcal{A}}$ denoting the lengths of the subintervals that define the transformation. We say that a collection of complex numbers $\{\zeta_\alpha\}_{\alpha \in \mathcal{A}}$ is a \emph{suspension data} for the linear involution if the following conditions are satisfied:
\begin{enumerate}
    \item For every $\alpha \in \mathcal{A}$, it holds that $Re(\zeta_\alpha) = \lambda_\alpha$.
    \item For every $1 \leq j < l$, 
    \[
    \sum_{\pi(i) \leq j} Im(\zeta_{\pi(i)}) > 0.
    \]
    \item For every $l+1 \leq j < l+m$, 
    \[
    \sum_{l < \pi(i) \leq j} Im(\zeta_{\pi(i)}) < 0.
    \]
    \item The following equality is satisfied:
    \[
    \sum_{\pi(i) \leq l} Im(\zeta_{\pi(i)}) = \sum_{l < \pi(i) \leq l+m} Im(\zeta_{\pi(i)}).
    \]
\end{enumerate}

Given a suspension data $\{\zeta_\alpha\}_{\alpha \in \mathcal{A}}$ for a linear involution $T(\pi,\lambda)$, we can define two piecewise linear segments, $L_0$ and $L_1$, as follows:\\

-The vertices of $L_0$ are determined by the complex numbers 
\[
\{0, \zeta_{\pi(1)}, \ldots, \sum_{\pi(i) \leq l} \zeta_{\pi(i)}\}.
\]
-The vertices of $L_1$ are determined by the complex numbers 
\[
\{0, \zeta_{\pi(l+1)}, \ldots, \sum_{l+1 \leq \pi(i) \leq l+m} \zeta_{\pi(i)}\}.
\]

One can verify that it is possible to identify parallel segments via translation. This identification yields a half-translation surface where the "first return" process described before to $\overline{X}$  by the vertical foliation coincides with the linear involution $T(\pi,\lambda)$.\\

Given a linear involution $T(\pi,\lambda)$ and a half-translation surface $S$ constructed from suspension data $\{\zeta_\alpha\}_{\alpha \in \mathcal{A}}$, observe that for any $\alpha \in \mathcal{A}$, the first return time of any point $x \in X_i$ under the vertical flow into $X_{\sigma(i)}$ is constant. This return time, denoted as $h_{\pi(i)}$, depends only on the generalized permutation $\pi$ and the suspension data $\{\zeta_\alpha\}_{\alpha \in \mathcal{A}}$.  \\

In particular, this implies that the rectangle $(0, \lambda_{\pi(i)}) \times (0, h_{\pi(i)})$ can be embedded into $S$. Thus, as in the case of translation surfaces, one can perform the so-called \emph{Veech zippered rectangle} construction. This construction provides an isometric representation of $S$ as a quotient of rectangles $\bigsqcup_{\alpha \in \mathcal{A}} R_{\alpha}$, where each rectangle is given by $R_\alpha = (0, \lambda_{\alpha}) \times (0, h_{\alpha})$. The identifications of the edges depend only on the linear involution $T(\pi,\lambda)$ and the generalized permutation $\pi$.  \\

For further details, see \cite{boissy2009dynamics} and Section 4 \cite{belldiagonal} for the case of linear involutions or \cite{viana2006ergodic} for a more detailed exposition in the case of interval exchange transformations. This representation of the surface $S$ provides a valuable framework for understanding the relationship between the Rauzy--Veech renormalization process and the \textit{lack} of weak mixing, as discussed in Lemma \ref{veechcriteria}. In this setting we can perform a renormalization process described by Rauzy--Veech matrices.\\

We are interested in considering linear involutions that arise from   half-translation surfaces and, moreover, for which Rauzy--Veech induction can be applied infinitely many times. We will present the conditions found by Boissy--Lanneau \cite{boissy2009dynamics} to ensure this.\\
 Now we will state a notion of \textit{reducibility} for generalized permutations. Consider an alphabet of $d$ letters $\mathcal{A}=\{\alpha_1,...,\alpha_d\}$ and a generalized permutation represented as

\[ \pi=
\left( \begin{array}{c|c|c}

\text{$A\cup B$} & \text{$***$} & \text{$B\cup D$} \\
\hline
\text{$A\cup C$} & \text{$***$} & \text{$D\cup C$} \\
\end{array} \right)
\]
where the sets $A,B,C,D$ are possible empty unordered sets of $\mathcal{A}$. We say that $\pi$ is \emph{reducible} if some of the following holds:
\begin{enumerate}
    \item There are no empty corner.
    \item There is exactly one empty corner and it is on the left.
    \item There are exactly two empty corners and both of them are on the right or on the left.
\end{enumerate}

We say that the generalized permutation $\pi$ is irreducible if it is not reducible. Theorem 3.2 in \cite{boissy2009dynamics} asserts that if $T(\pi,\lambda)$ is a linear involution, then there exists a suspension data if and only if the generalized permutation $\pi$ is irreducible. \\

This \textit{geometric irreducibility} is not enough, since this does not imply that we can apply Rauzy--Veech induction infinitely many times, therefore we will state what Boissy-Lanneau refer as \emph{dynamically irreducible.}\\

\textbf{Notation:} For a linear involution $T(\pi,\lambda)$  on an alphabet $\mathcal{A}$, we consider a partition on it such that \begin{itemize}
    \item $\mathcal{A}_{01}$ is the set such that for every \( \pi(i) \in \mathcal{A}_{01} \), the intervals \( X_{\pi(i)} \) and \( X_{\sigma(i)} \) belong to different connected components of \( \overline{X} \).
    \item The set \( \mathcal{A}_0 \) satisfies that for every \( \pi(i) \in \mathcal{A}_0 \), the intervals \( X_{\pi(i)} \) and \( X_{\sigma(i)} \) are contained in \( X \times \{0\} \).
    \item The set \( \mathcal{A}_1 \) satisfies that for every \( \pi(i) \in \mathbf{A}_1 \), the intervals \( X_{\pi(i)} \) and \( X_{\sigma(i)} \) are contained in \( X \times \{1\} \).
\end{itemize}

If $T(\pi,\lambda)$ is a linear involution on $X\times\{0,1\}$ we say that $(x,\epsilon)\in X\times \{0,1\}$ is a \textit{connection} of length $n$ if $(x,\epsilon)$ is a singularity for $T^{-1}$ and $T^n(x,\epsilon)$ is a singularity for $T$. We say that $T$ satisfies \textit{ the Keane} condition when it does not have connections of any length. This is equivalent to say that the Rauzy--Veech induction $\mathcal{R}^n(\pi,\lambda)$ is always defined (Proposition 4.2 \cite{boissy2009dynamics}).
We define the length parameter \(\lambda\) of a linear involution \(T(\pi, \lambda)\) as \emph{admissible} if the following conditions are not satisfied:

\begin{enumerate}
    \item The generalized permutation \(\pi\) can be decomposed as any of the following:
    \[
    \pi = \left( \begin{array}{c|c}
    A & \text{***} \\
    \hline
    A & \text{***}
    \end{array}\right), \quad
    \pi =\left( \begin{array}{c|c}
    D & \text{***} \\
    \hline
    D & \text{***}
    \end{array}\right), \quad
    \pi = \left( \begin{array}{c|c}
    A \cup B & B \cup D \\
    \hline
    A \cup C & C \cup D
    \end{array}\right).
    \]
    Here:
    \begin{itemize}
        \item \(A, D \subseteq \mathcal{A}_{0,1}\), where \(A\) and \(D\) are non-empty in the first two cases.
        \item \(B = \mathcal{A}_0\), \(C = \mathcal{A}_1\).
    \end{itemize}
    
    \item The generalized permutation \(\pi\) can be decomposed as:
    \[
    \pi = \left(
    \begin{array}{c|c|c}
    A \cup B & \text{***} & B \cup D \\
    \hline
    A \cup C & \beta, \text{***}, \beta & D \cup C
    \end{array}\right),
    \]
    where:
    \begin{itemize}
        \item \(A, D \subseteq \mathcal{A}_{0,1}\),
        \item \(\emptyset \neq B \subseteq \mathcal{A}_0\),
        \item \(C \subseteq \mathcal{A}_1\).
    \end{itemize}
    
    Furthermore, the length parameters \(\{\lambda_\alpha\}_{\alpha \in \mathcal{A}}\) must satisfy the inequality: \[
    \sum_{\alpha \in A} \lambda_\alpha \leq \sum_{\alpha \in B} \lambda_\alpha + \lambda_\beta + \sum_{\alpha \in C} \lambda_\alpha.
    \]
\end{enumerate}
 We say that a linear involution is \emph{dynamically irreducible} if the set of admissible parameters is non-empty. Note that this set is always open. Theorem B in \cite{boissy2009dynamics} states that for dynamically irreducible linear involutions, Rauzy--Veech induction is always well defined. Moreover, from Theorem 12.1 and Proposition 6.2 in \cite{gadre2012dynamics} we can deduce that for a full-measure subset of admissible parameters, the linear involution is uniquely ergodic. This almost certain ergodicity for admissible parameters will play a crucial role in  Lemma \ref{veechcriteria}.
\\
 \subsection{Measurable Cocycles}
We adopt the terminology and framework introduced by Avila--Forni in \cite{avila2007weak} for the following discussion.

Let $T$ be a measurable transformation on a probability space $(\Delta, \mu)$. We say that $T$ is \emph{weakly expanding} if there exists a partition of $\Delta$ (modulo zero) into a collection $\{\Delta^{(l)} : l \in \mathbb{Z}\}$ such that, for each $l \in \mathbb{Z}$, the restriction
\[
T^{(l)} := T_{|\Delta^{(l)}} : \Delta^{(l)} \to \Delta
\]
is invertible and the pushforward measure $T^{(l)}_*\mu$ is equivalent to $\mu$.

For any finite word of integers $w = (l_1, \dots, l_n)$, define
\[
\Delta^{(w)} := \{x \in \Delta : T^{k-1}(x) \in \Delta^{(l_k)}, \ 1 \leq k \leq n\},
\]
with the corresponding restriction $T^w := T_{|\Delta^{(w)}}$. Let $\Omega$ denote the set of all finite words. Then $T$ is said to be \emph{strongly expanding} if there exists a constant $K>0$ such that, for every
\[
\nu \in \mathcal{M} := \left\{ \nu = \frac{T_*^{(w)} \mu}{\mu(\Delta^{(w)})} : w \in \Omega \right\},
\]
we have
\[
\frac{1}{K} \leq \frac{d\nu}{d\mu} \leq K.
\]
This is also refered as \emph{bounded distortion property}.
The Hilbert metric on the projective space \(\mathbb{P}^{d-1}\) is defined by
\[
d(x, y) = \sup_{1 \leq i, j \leq d} \left| \log \frac{x_i y_j}{x_j y_i} \right|.
\]
  We define the \emph{standard simplex}, denoted by $\mathbb{P}_+^{d-1}$, as the projectivization of the positive cone $\mathbb{R}_+^d$. 
For a subset \(A \subset \mathbb{P}^{d-1}_+\), the Hilbert metric is well defined if \(A\) is convex, contains a non-empty open set, and does not contain a proper linear subspace.

A transformation is called a \emph{projective contraction} if it arises from the projectivization of a matrix in $GL(d, \mathbb{R})$ with non-negative entries. This ensures that the transformation maps the standard simplex into itself.  

Given a hyperplane $H$, consider its intersection with the standard simplex, $H \cap \mathbb{P}_+^{d-1}$. If there exists a projective contraction $T$ mapping this intersection into itself, we call the image
\[
T(H \cap \mathbb{P}_+^{d-1})
\]
a \emph{generalized simplex}.

The following technical lemma, derived in \cite{avila2007weak}, was originally stated for transformations defined on a simplex. The proof carries over verbatim for generalized simplices. It provides a sufficient condition for a transformation to be strongly expanding:

\begin{lemma}[Lemma 2.1, \cite{avila2007weak}] \label{AFboundeddistortion} 
Let \(\Delta\) be a generalized simplex compactly contained in the intersection of the standard simplex and a hyperplane $H$ such that the Hilbert metric is well defined, and let \(\{\Delta^{(l)} : l \in \mathbb{Z}\} (\mathrm{mod} \, 0)\) be a partition of \(\Delta\) where each \(\Delta^{(l)}\) has positive Lebesgue measure. Suppose \(T : \Delta \to \Delta\) is a measurable transformation such that \(T(\Delta^{(l)}) = \Delta\), and \(T^{(l)} = T_{|\Delta^{(l)}}\) is invertible, with \((T^{(l)})^{-1}\) being the restriction of a projective contraction. Then \(T\) preserves a measure \(\mu\) that is absolutely continuous with respect to the Lebesgue measure, with a density that is a positive continuous function on \(\overline{\Delta}\). Moreover, \(T\) is strongly expanding with respect to \(\mu\).  
\end{lemma}  

\subsubsection{Cocycles} 
Cocycles naturally arise in dynamical systems to describe how additional structure evolves along trajectories of a transformation. Let $T : \Delta \to \Delta$ be a measurable transformation. A \emph{cocycle} is a pair $(T, A)$, where $A : \Delta \to SL(d, \mathbb{R})$ assigns a matrix in the special linear group to each point of $\Delta$. The cocycle governs the evolution of vectors in $\mathbb{R}^d$:
\[
(T, A)(x, v) = (T(x), A(x) \cdot v), \quad \text{for } (x, v) \in \Delta \times \mathbb{R}^d.
\]

Iterating the cocycle $n$ times gives
\[
(T, A)^n(x, v) = (T^n(x), A^n(x) \cdot v),
\]
where
\[
A^n(x) = A(T^{n-1}(x)) \cdots A(x)
\]
is the product of matrices along the trajectory of $x$.

If $(\Delta, \mu)$ is a probability space and $T$ is ergodic with respect to $\mu$, we say the cocycle is \emph{measurable} if
\[
\int_\Delta \log \|A(x)\| \, d\mu(x) < \infty.
\]
A stronger condition defines a \emph{uniform cocycle}, requiring
\[
\int_\Delta \log \max\{\|A(x)\|, \|A^{-1}(x)\|\} \, d\mu(x) < \infty.
\]

These integrability conditions allow us to define invariant structures associated with the cocycle, which are central to its dynamical analysis. In particular, for measurable cocycles, the long-term behavior of vectors under iteration is captured by the growth of $\|A^n(x)\cdot v\|$. This leads to the definition of key invariant subspaces:

\begin{itemize}
    \item The \emph{stable space}:
    \[
    E^s(x) = \{v \in \mathbb{R}^d : \lim_{n \to \infty} \|A^n(x)\cdot v\| = 0\}.
    \]

    \item The \emph{central stable space}:
    \[
    E^{cs}(x) = \{v \in \mathbb{R}^d : \limsup_{n \to \infty} \|A^n(x)\cdot v\|^{1/n} \leq 1\}.
    \]

    \item The \emph{weak stable space}:
    \[
    W^s(x) = \left\{ v \in \mathbb{R}^d : \lim_{n \to \infty} \|A^n(x)\cdot v\|_{\mathbb{R}^d/\mathbb{Z}^d} = 0 \right\},
    \]
    where $\|\cdot\|_{\mathbb{R}^d/\mathbb{Z}^d}$ denotes the distance in the quotient space $\mathbb{R}^d/\mathbb{Z}^d$.
\end{itemize}

These subspaces satisfy the following equivariance properties:
\[
A(x)(E^s(x)) = E^s(T(x)) \quad \text{and} \quad A(x)(E^{cs}(x)) = E^{cs}(T(x)), \quad \text{for almost every } x \in \Delta.
\]

Cocycles can be further classified according to their structure. The cocycle $(T, A)$ is \emph{locally constant} if $T$ is strongly expanding and, for each $l \in \mathbb{Z}$, the restriction $A_{|\Delta^{(l)}}$ is constant. If, in addition, $A(x) \in SL(d, \mathbb{Z})$ for almost every $x$, the cocycle is said to be \emph{integral}.

Finally, consider compact subsets $\Theta \subseteq \mathbb{P}^{d-1}$ of the projective space. We say that $\Theta$ is \emph{adapted} to the cocycle $(T, A)$ if, for every nonzero vector $w \in \mathbb{R}^d$ whose projectivization lies in $\Theta$, the following hold:
\begin{itemize}
    \item $\|A^n(x)\cdot w\| \geq \|w\|$ for all $n \geq 0$,
    \item $\lim_{n \to \infty} \|A^n(x) \cdot w\| = \infty$.
\end{itemize}

The set of lines in $\mathbb{R}^d$ not parallel to elements of $\Theta$ is denoted by $\mathcal{J}(\Theta)$. Adapted sets are a useful tool for studying the growth behavior of cocycles and their geometric implications.

\subsection{Kontsevich-Zorich Cocycle}

Consider the unit area strata of orientable quadratic differentials $\mathcal{M}^{(1)}_\kappa$ and the diagonal flow $g_t$. The Kontsevich-Zorich cocycle $G_t$ is defined as the quotient cocycle
\[
 g_t \times \text{id} : \mathcal{M}_\kappa^{(1)} \times H^1(S, \mathbb{R}) \to \mathcal{M}_\kappa^{(1)} \times H^1(S, \mathbb{R})
\]
quotiented by the mapping class group $\Gamma_g$. We can define the same cocycle for the non-orientable quadratic differentials restricting to $\mathcal{Q}_\kappa^{(1)}$.\\

 We take the Kontsevich-Zorich cocycle, since this cocycle is log-integrable with respect to the Lebesgue measure (see \cite{kontsevich1997lyapunov}), the Oseledets multiplicative ergodic theorem implies that for $\mu$-almost every point $q \in \mathcal{Q}_\kappa^{(1)}$ and every non-zero vector $v \in H^1(S, \mathbb{R})$, if $\mathcal{K}$ is the action on the fibers the following limit exists:
\[
\lim_{n \to \infty} \frac{1}{n} \log \frac{\|\mathcal{K}^n(q) v\|}{\|v\|}.
\]

   The possible values for this limit, counted with multiplicity, which are known as the \emph{Lyapunov exponents}  describe the exponential rate of expansion (positive values) or contraction (negative values). The Lyapunov exponents satisfy the following inequalities:
\[
1 = \lambda_1 \geq \lambda_2 \geq \cdots \geq \lambda_g \geq 0 \geq \lambda_{g+1} \geq \cdots \geq \lambda_{2g} = -1.
\]

Moreover, there exists a decomposition of $H^1(S, \mathbb{R})$:
\[
H^1(S, \mathbb{R}) = E_1 \supseteq E_2 \supseteq \cdots \supseteq E_{2g} \supseteq \{0\},
\]
such that for every $v \in E_i \setminus E_{i+1}$, the following holds:
\[
\lim_{n \to \infty} \frac{1}{n} \log \frac{\|\mathcal{K}^n(q) v\|}{\|v\|} = \lambda_i.
\]

The dimension of $E_i$ corresponds to the multiplicity of $\lambda_i$. Additionally, since the Kontsevich-Zorich cocycle is symplectic, it follows that $\lambda_{-i} = -\lambda_i$. \\

Consider a quadratic differential \(\omega\) obtained from a non-orientable quadratic differential \(q \in \mathcal{Q}_\kappa^{(1)}\) using the double cover construction. The Kontsevich-Zorich cocycle can be defined on \(\mathcal{M}_{\hat{\kappa}}^{(1)}\) and applied to \(\omega\), with the action restricted to the splitting  
\[ H^1(R, \mathbb{R}) = H^+(R) \oplus H^-(R). \]  
It can be verified that the cocycle preserves this splitting. Consequently, on each subspace, one defines a measurable symplectic cocycle, leading to the existence of Lyapunov exponents for the two components.  \\

For the \(H^+(R, \mathbb{R})\) subspace, the Lyapunov exponents are  
\[
\lambda^+_1 \geq \cdots \geq \lambda^+_g \geq 0 \geq \lambda^+_{g+1} \geq \cdots \geq \lambda^+_{2g},
\]  
and for the \(H^-(R, \mathbb{R})\) subspace, they are  
\[
\lambda^-_1 \geq \cdots \geq \lambda^-_{g+s/s-1} \geq 0 \geq \lambda^-_{g+s/s} \geq \cdots \geq \lambda^-_{2g+s-2}.
\]  

By the work of R. Treviño (see \cite{trevino2013non}), for every connected component of a stratum, it holds that \(\lambda_g^+ > 0\) and \(\lambda_{g+s/2-1}^- > 0\). We are  interested in the restriction of the Kontsevich-Zorich cocycle to the minus part of the splitting. \\

In Section 4 of \cite{belldiagonal}, there is a description of the zippered rectangle construction and introduce what they call \emph{singularity parameters} along with curves \(c_\alpha\) designed for each rectangle \(R_\alpha\). These curves are chosen such that their homology classes span the minus part of the splitting. Moreover, with respect to this basis, the matrix representing the change induced by the Kontsevich-Zorich cocycle is precisely the Rauzy--Veech matrix (see Section 7.11 of \cite{belldiagonal}).  \\

The central result of \cite{belldiagonal} is summarized in the following theorem:  
\begin{thm}  \label{ThmZariskiDense}
    (\cite{belldiagonal})  
    The plus and minus symplectic Rauzy--Veech groups associated with every connected component of a stratum of quadratic differentials are Zariski dense in their respective ambient symplectic spaces.  
\end{thm}  

This result has significant implications for the study of the weak mixing property of linear involutions arising from half-translation surfaces. By applying the Avila-Viana criterion for Lyapunov simplicity \cite{avila2007simplicity}, it follows that the Rauzy--Veech cocycle for half-translation surfaces is simple, meaning that all Lyapunov exponents are pairwise distinct. This simplicity allows for a precise determination of the dimensions of the stable and central stable spaces of the cocycle. Furthermore, it ensures the existence of both \emph{pinching and twisting} matrices within the associated symplectic Rauzy--Veech groups.  \\
\subsection{Rauzy-Veech Cocycle}

For any irreducible generalized permutation we can consider its parameter space and intersect it with $\mathbb{P}_+^{d-1}$, and define this intersection as $W_\pi$. Observe that Rauzy-Veech induction commutes with dilations, namely for every $r>0$ it is true that
$$\mathcal{Q}(\pi,r\lambda)=(\pi^{(1)},r\lambda^{(1)}),$$
this implies we can define  a projectivization of this map called Rauzy renormalization
$$\mathcal{R}:\{\pi\}\times W_\pi\to \{\pi^{(1)}\}\times W_{\pi^{(1)}}.$$
     Consider a path by Rauzy--Veech moves
    $$(\pi,\lambda)\to (\pi^{(1)},\lambda^{(1)})\to ...\to (\pi^{(n)},\lambda^{(n)}).$$
 We can forget for one moment what the induction does on parameters $\lambda$ and observe the change it generates to the generalized permutation $$\pi \underset{\gamma_1}{\rightarrow}\pi^{(1)}\rightarrow...\underset{\gamma_n}{\rightarrow}\pi^{(n)}$$
    and associate to the path concatenation $\gamma=\gamma_1\cdot ...\cdot \gamma_n$ a matrix $B_\gamma=B_{\gamma_n}\cdot...\cdot B_{\gamma_1}$.\\

 Take $\pi$ a generalized permutation, and a cycle $\gamma_0:\pi \to...\to \pi$ such that all the entries of the associated matrix $B_{\gamma_0}$ are positive (\cite{gadre2012dynamics}). Define $\Delta_\pi\subset W_\pi$ as the projectivization of $B_{\gamma_0}(W_\pi)$. Consider the set of all primitive cycles (that is, $\pi$ only appears once) $\gamma$ for $\pi$ and its matrices $B_\gamma$. This defines a countable partition of $\Delta_\pi$ such that for any cycle $\gamma:\pi\to...\to\pi$ defining $\Delta^{(\gamma)}$ as the projectivization of $B_\gamma(\Delta_\pi)$. This is the set of parameters in $\Delta_\pi$ such that they return to $\Delta_\pi$ with the renormalization following the cycle $\gamma$. And observe that the projectivization of $(B_\gamma)^{-1}(\Delta_\gamma) $ is $\Delta_\pi$.
 Define $T:\Delta_\pi\to \Delta_\pi$ as: if $\lambda\in \Delta_\gamma$ then $$T(\lambda)=\frac{(B_\gamma)^{-1}\cdot \lambda}{||(B_\gamma)^{-1}\cdot \lambda||}.$$
 We  define $A:\Delta_\pi\to SL(d,\mathbb{R})$, given by $A_{|\Delta_\gamma},\lambda\to B_\gamma^*$.
  Therefore defines a skew linear product on $\Delta_\pi\times \mathbb{R}^d$
    $$(T,A)(\lambda,w)=(\frac{B_\gamma^{-1}\cdot \lambda}{||B_\gamma^{-1}\cdot \lambda||},B_\gamma^*\cdot w).$$
     This satisfies $$(T,A)^n=(T^n,A(T^{n-1})\cdot...\cdot A).$$
     This defines a cocycle called \emph{Rauzy--Veech cocycle}. We define the set $\Delta_\pi^n$ as the subset of $\Delta$ on which $T^n$ is well-defined. The set of full measure is then given by
\[
\Delta_\pi^\infty \coloneqq \bigcap_{n \geq 0} \Delta_\pi^n.
\]

 \begin{lemma}
     There exists a measure $\mu$ on $\Delta$ which is $T$-invariant and absolutely continuous respect Lebesgue which has bounded distortion and such that the cocycle $(T,A)$ is uniform.
 \end{lemma}
     \begin{proof}
         The bounded distortion property can be checked by Lemma \ref{AFboundeddistortion}. It is relevant to see that every $W_\pi$ is a generalized simplex where you can define de Hilbert metric, this is treated in detail in Appendix D \cite{belldiagonal}. Another reference to check this fact is a  consequence of Lemma 6.18 in \cite{belldiagonal}. The fact that the cocycle $(T,A)$ is uniform follows from Lemma 6.39 in \cite{belldiagonal}.
     \end{proof}
        For an irreducible generalized permutation $\pi$ in an alphabet of $d$ letters there exists a symplectic subspace $H(\pi)$ of $\mathbb{R}^d$ with dimension $2g+n-2$ where $g$ is the genus of the half-translation surface that it generates and $n$ is the number of \textit{non-orientable} singularities.
        If  $\gamma$ is a cycle $\pi\to...\to\pi$ the matrix $$(B_\gamma)^*_{|H(\pi)}:H(\pi)\to H(\pi)$$ 
        is a symplectic isomorphism, see for example Section 7.11 \cite{belldiagonal}.\\

      Considering a double orientation cover it can be seen that $H(\pi)\subseteq H^-(R,\mathbb{R})$. If $k$ is the number of intervals for IETs that we can get from the double orientation cover,  in \cite{veech1984metric} we find that the Rauzy--Veech cocycle (for IETs) in $\mathbb{R}^k/H^+(R,\mathbb{R})\oplus H^-(R,\mathbb{R})$ is isometric, from this we conclude  that the Rauzy--Veech cocycle (for linear involutions) in $\mathbb{R}^d/ H(\pi)$ is also isometric, therefore the \textit{interesting} action of the Rauzy--Veech cocycle occurs in $H(\pi)$. Namely,  by Oseldets theorem for $\mu$-almost any $\lambda\in W_\pi$ and every non-zero $w\in H(\pi)$
    $$\lim_{n\to\infty}\frac{1}{n}\log \frac{||A_n(\lambda)\cdot w||}{||w||}$$ exists.
     In particular there exists Lyapunov exponents
    $$\lambda_{1}\geq ...\geq\lambda_{g+n/2-1}\geq 0\geq \lambda_{g+n/2}\geq...\geq \lambda_{2g+n-2}.$$
    
 The relationship between the Rauzy--Veech cocycle and the Kontsevich--Zorich cocycle for half-translation surfaces is explained in Section~7.16 of \cite{belldiagonal}. In Section 6 \cite{belldiagonal}, the authors study a suspension flow over the Rauzy--Veech cocycle such that, for any irreducible generalized permutation $\pi$, the set $\Delta_\pi \times \{\pi\}$ forms a cross-section to this flow. They investigate the return time function $r_\pi$ associated to this section, and consider the average return time with respect to a measure $\mu$ for which the cocycle has strong distortion properties:
\[
r_{\mathrm{av}} = \int_{\Delta_\pi} r_\pi(\lambda) \, d\mu.
\]
In \cite{belldiagonal}, it is shown that the Lyapunov exponents of the cocycle for linear involutions are related to the negative Lyapunov exponents of the Kontsevich--Zorich cocycle by the identity
\[
\lambda_i^{\mathrm{LI}} = r_{\mathrm{av}} \cdot \lambda_i^-.
\]

As a consequence, the Rauzy--Veech cocycle for linear involutions is also simple, meaning that all its Lyapunov exponents are pairwise distinct.  
 Moreover as consequence of Theorem \ref{ThmZariskiDense} the monoid of matrices that form the Rauzy--Veech cocycle restricted to $H(\pi)$ is Zariski dense in the ambient space.

\section{Proof of Main Result}
Recall that for a transformation \(T : X \to X\), weak mixing can be characterized with spectral properties: the associated Koopman operator admits no non-constant eigenfunctions, and  the only eigenfunctions correspond to the eigenvalue \(1\). This is often described as having purely continuous spectrum.

For linear involutions, this can be stated as follows: the map \(T(\pi, \lambda)\) is weakly mixing if there is no non-constant measurable function \(f : \overline{X} \to \mathbb{C}\) such that, for some \(v \notin \mathbb{Z}^d\),
\begin{equation} \label{eigenfunction}
    f(T(x)) = e^{2\pi i v_j} f(x)
\end{equation}
holds for every \(x \in X_\alpha \cup X_{\sigma(\alpha)}\) and all \(\alpha \in \mathcal{A}\).

If \(T(\pi, \lambda)\) is ergodic but not weakly mixing, then there exists a non-zero measurable function \(f : \overline{X} \to \mathbb{C}\) and a scalar \(t \notin \mathbb{Z}\) such that
\begin{equation} \label{ergodicnonweakmixing}
    f(T(x)) = e^{2\pi i t} f(x)
\end{equation}
for all \(x \in X_\alpha \cup X_{\sigma(\alpha)}\) and every \(\alpha \in \mathcal{A}\).

The following result generalizes Lemmas 7.2--7.3 in Veech~\cite{veech1984metric}, and we refer to it as the \emph{Veech criterion for linear involutions}. Although its validity follows from standard arguments based on Luzin’s theorem, we present a proof following the strategy in Veech’s proof of Lemma 7.2 for interval exchange transformations, for expository clarity. This lemma serves as a bridge between vectors satisfying~\eqref{ergodicnonweakmixing} and the weak stable space of the Rauzy--Veech cocycle.

\begin{lemma} \label{veechcriteria}
    Let $\pi$ be a dynamically irreducible generalized permutation on an alphabet $\mathcal{A}$ of $d$ elements.  For a.e $\lambda\in W_\pi$ , if there exists an neighborhood of  $\mathcal{U}_{\mathcal{R}}\subseteq W_\pi$  of $\lambda$ and  an infinite set $E\subseteq \mathbb{N}$ such that for every $n\in E$ it is true that $\mathcal{R}^n(\pi,\lambda)\in \mathcal{U}_{\mathcal{R}}$ and
    if there exists a non trivial measurable function $f:\overline{X}\to \mathbb{C}$ such that  there exists a $v\in \mathbb{R}^{d}$ such that \begin{equation} \label{solution}
         f(T(x))=e^{2\pi i v_\alpha}f(x),
    \end{equation}
    for every $x\in X_{\alpha},X_{\sigma(\alpha)}$ and for every $\alpha\in \mathcal{A}$ then 
    $$\lim_{n\in E\to \infty}||(B^n)^*(\pi,\lambda)\cdot v||_{\mathbb{R}^{d}/\mathbb{Z}^{d}}=0.$$
\end{lemma}
\begin{proof}
 Without loss of generality, we may assume that $T$ is an ergodic linear involution \cite{gadre2012dynamics}. Suppose that $T$ satisfies the statement in Lemma \ref{veechcriteria}. Let $f$ be a measurable solution to 
\[
f(T(x)) = e^{2\pi i v_\alpha} f(x), \quad \text{for every } x \in X_\alpha, X_{\sigma(\alpha)}.
\]

We use the fact that the renormalization process for linear involutions is ergodic (Theorem 13.1 in \cite{gadre2012dynamics}). Then, it is possible to show that there exists an infinite set $E \subseteq \mathbb{N}$ and a constant $\epsilon > 0$ such that for every $n \in E$, the following hold:

\begin{itemize}
    \item There exists $k_n \geq \frac{\epsilon |\lambda|}{|\lambda^{(n)}|}$ such that for every $k \leq k_n$, the iterates $T^k(\overline{X^{(n)}})$ consist of at most two intervals sharing at most one extremal point.
    
    \item For every $j$, it holds that $|X_j^{(n)}| \geq \epsilon |\lambda^{(n)}|$.
\end{itemize}

It can be deduced from Section 10 \cite{gadre2012dynamics} that these distortion conditions are satisfied if, along the Rauzy--Veech path
\[
(\pi, \lambda) \to (\pi^{(1)}, \lambda^{(1)}) \to \cdots \to (\pi^{(n)}, \lambda^{(n)}),
\]
there exists a stage where the matrix $B^n(\pi, \lambda)$ has all entries strictly positive and the corresponding combinatorial type $\pi^{(n)}$ recurs infinitely often.

To ensure this, one can choose a path $\gamma_0$ of length $|\gamma_0| = n$ sufficiently large so that all entries of $B_{\gamma_0}$ are positive and $\Delta_\pi \subseteq \mathcal{U}_{\mathcal{R}}$. Define the Rauzy--Veech cocycle over $\Delta_\pi$ as explained before.

Consider the full measure subset $\Delta_\pi^\infty$ consisting of elements whose Rauzy--Veech path includes a return to $\Delta_\pi$ along a cycle $\gamma_0$ with strictly positive matrix $B_{\gamma_0}$, and for which the generalized permutation $\pi^{(\gamma_0)}$ occurs infinitely often. This implies that the set of parameters in $\Delta_\pi$ for which $T(\pi, \lambda)$ fails to satisfy these conditions has zero measure.

Since Rauzy--Veech induction is expanding, every point in $W_\pi$ eventually enters $\Delta_\pi$ under forward iteration. Therefore, the same measure bound holds for the full parameter space $W_\pi$, and we observe that
\[
\Big|\bigcup_{j=0}^{k_n-1} T^j(\overline{X^{(n)}})\Big| \geq \epsilon |\lambda|.
\]

By ergodicity of $\mathcal{R}$, as $n \to \infty$ in $E$, we have $|\lambda^{(n)}| \to 0$. Also, by ergodicity of $T$ and since $k_n \to \infty$, the sets
\[
\bigcup_{k=0}^{k_n} T^k(\overline{X^{(n)}})
\]
become equidistributed in $\overline{X}$. By Luzin's theorem, $f$ can be approximated on compact sets by continuous functions. Hence, there exists $n$ sufficiently large such that for some $k \leq k_n$, there exist constants $z_k^1, z_k^2$ satisfying
\[
|f(x) - z_k^i| < \delta, \quad \text{for all } x \in X^{(n)} \times \{i\}.
\]
Consequently,
\[
\int_{T^k(X^{(n)} \times \{i\})} |f(x) - z_k^i| \, d\mu < \delta |\lambda^{(n)}|.
\]

Recall that $T^k(X_i^{(n)})$ are intervals. By the cocycle property of $f$ with respect to $T$, we have
\[
f(T^k(x)) = \prod_{m=0}^{k} e^{2\pi i j(m)} f(x),
\]
where $j(m)$ is the nonzero entry of the vector $w(m) \in \mathbb{R}^d$ defined by $w_l(m) = v_l \chi_{X_l}(T^m(x))$. It follows that $\prod_{m=0}^k e^{2\pi i j(m)}$ is constant, so there exist two complex numbers $z_0, z_1$ such that
\[
\int_{X^{(n)} \times \{i\}} |f(x) - z_i| \, d\mu < \delta |\lambda^{(n)}|.
\]
By Chebyshev's inequality, we obtain
\[
\mu\big(\{x \in X^{(n)} \times \{i\} : |f(x) - z_i| \geq \sqrt{\delta}\}\big) \leq \sqrt{\delta} |\lambda^{(n)}|.
\]

Using the interpretation of the Veech zippered rectangle construction for first return times, consider the renormalization matrix $B^n(\pi, \lambda)$ and the dual matrix $(B^n(\pi, \lambda))^*$ as the visiting matrix. For each $j \leq d$, define
\[
(B_j^n)^* = \sum_{i=1}^d (B^n)^*_{ij},
\]
and let $r_{j,n} \geq (B_j^n)^*$ such that for every $x \in X_j^{(n)} \subseteq X^{(n)} \times \{i\}$, we have
\[
T^{r_{j,n}}(x), \, T^{r_{j,n} - (B_j^n)^*}(x) \in X^{(n)} \times \{i\}, \quad i \in \{0,1\}.
\]

Taking $\delta < \frac{\epsilon^2}{4}$, for each $1 \leq j \leq d$, there exists $x \in X_j^{(n)}$ such that
\[
|f(T^{r_{j,n} - (B_j^n)^*}(x)) - z_i|, \, |f(T^{r_{j,n}}(x)) - z_i| < \sqrt{\delta}.
\]

Such $x$ must exist; otherwise,
\[
\epsilon |\lambda^{(n)}| \leq |X_j^{(n)}| \leq \mu\Big(\{x \in X_j^{(n)} : |f(T^{r_{j,n}-(B_j^n)^*}(x)) - z_i| \geq \sqrt{\delta}\}\Big) + \mu\Big(\{x \in X_j^{(n)} : |f(T^{r_{j,n}}(x)) - z_i| \geq \sqrt{\delta}\}\Big)\] \[ < 2 \sqrt{\delta} |\lambda^{(n)}| < \epsilon |\lambda^{(n)}|.
\]

Finally, by ergodicity of $T$, we may assume that almost everywhere $|f| = 1$. By the triangle inequality, this implies
\[
|e^{2\pi i (v (B^n)^*_j)} - 1| \leq 2 \sqrt{\delta},
\]
which concludes the proof.

\end{proof}
The following theorem is a generalization of Theorem 5.1 in Avila--Forni \cite{avila2007weak}. It is a consequence of the fact that the monoid generated by the Kontsevich-Zorich cocycle (and also the Rauzy--Veech cocycle) is Zariski dense, and therefore it is \textit{pinching} and \textit{twisting}.
\begin{thm} \cite{avila2007weak}\label{avilaf}
    Let $\pi$ be a generalized permutation, such that it codes the vertical foliation on a  half-translation surface of $g>1$. Let $L\subseteq H(\pi)$ be a line not passing through zero. If the central stable space of the Rauzy--Veech cocycle $E^{cs}$ is such that $dim(E^{cs})<2g +n/2-3$ then for almost any $[\lambda]\in W_\pi$, it is true that $L\cap E^{cs}([\lambda],\pi)=\emptyset$.
\end{thm}

\begin{proof}
    In \cite{avila2007weak} the authors consider Rauzy--Veech matrices $B_{(i)}$  that send the direction vector of $L$ to independent directions deriving a contradiction with the dimension assumption of $E^{cs}$. In fact these matrices are \textit{twisting} matrices arising from the fact that the Rauzy--Veech monoid is Zariski dense. Similarly, if a monoid of matrices is Zariski dense then there exists a finite set of (twisting) matrices $\mathcal{F}$ such that for every $k<d$ and every subspaces $V_1,V_2$ of dimensions $k,d-k$ respectively, there exists an $F\in \mathcal{F}$ such that $F(V_1)\cap V_2=\{0\}$. Since the minus part of the Kontsevich-Zorich cocycle is Zariski dense \cite{belldiagonal}  there exists twisting matrices that send the direction of an affine line $L$ to independent directions, this would imply that the codimension of the central stable space in $H(\pi)$ is at most 1, which contradicts the fact that $dim(E^{cs}([\lambda],\pi))=g+n/2-1>1$ for almost any $[\lambda]$.
\end{proof}
The fact that $dim(E^{cs}([\lambda],\pi))=g+n/2-1$ follows from simplicity and non uniform hyperbolicity of the Rauzy--Veech cocycle.\\

The next result is a specialization of Theorem~3.1 in \cite{avila2007weak}, whose proof carries over verbatim. It shows that non-trivial solutions to \eqref{ergodicnonweakmixing} occur only for a set of parameters of measure zero.

\begin{lemma}[\cite{avila2007weak}] \label{avilac}
    Let \((T,A)\) be a locally constant integral uniform cocycle defined over a generalized simplex where the Hilbert metric is well defined, and let \(\Theta\) be a set adapted to \((T,A)\). If for every \(J \in \mathcal{J}(\Theta)\) we have \(J \cap E^{cs}(x) = \emptyset\) for a.e.\ \(x \in \Delta\), then for every line \(L \subset \mathbb{R}^p\) parallel to some element of \(\Theta\) we have \(L \cap W^s(x) \subseteq \mathbb{Z}^p\) for a.e.\ \(x \in \Delta\).
\end{lemma}

In \cite{belldiagonal}, it is shown that the return times of a flow over the Rauzy--Veech cocycle have exponential tails, which implies that the cocycle is uniform. Combined with Lemma~\ref{AFboundeddistortion}, whose hypotheses are satisfied by the Rauzy--Veech cocycle, this allows us to apply Lemma~\ref{avilac}.

\begin{proof} 
 \textit{[Theorem \ref{theorem2}]}

Let \(\pi\) be a dynamically irreducible permutation on \(d\) intervals such that, for admissible parameters \(\lambda\), the map \(T(\pi,\lambda)\) codes the vertical foliation of a half-translation surface of genus \(g>1\).

Suppose there exists a positive measure set of parameters \(\lambda \in W_\pi\) for which \(T(\pi,\lambda)\) is ergodic but not weakly mixing. For each such \emph{bad parameter} \(\lambda\), there exists a non-trivial measurable function \(f : \overline{X} \to \mathbb{C}\) and a scalar \(t \notin \mathbb{Z}\) such that \(f\) and \(v = (t, \dots, t)\) satisfy Equation~\eqref{ergodicnonweakmixing} for \(T(\pi,\lambda)\).

\medskip
\noindent\textbf{Case 1:} Assume that \((1, \dots, 1) \notin H(\pi)\) for a positive measure subset of bad parameters.
Consider the orthogonal complement \(H(\pi)^\perp\). From the descriptions of \(H^{-}(R,\mathbb{R})\) in Section 7.11 of \cite{belldiagonal} or in \cite{gutierrez2019classification}, one can construct a basis of \(H(\pi)^\perp\) consisting of integer vectors \(w_1, \dots, w_{d + 1 - g - n/2}\).

By Lemma~\ref{veechcriteria} and the fact that the action of the Rauzy--Veech cocycle  on \(H(\pi)^\perp\) \cite{veech1984metric} is isometric, it follows that for any such basis element \(w\) and any vector \(v\) satisfying Equation~\eqref{solution}, we must have \(w \cdot v \in \mathbb{Z}\).

Since \((1, \dots, 1) \notin H(\pi)\), there exists \(w\) in the basis of  \( H(\pi)^\perp\) such that \(w \cdot (1, \dots, 1) \neq 0\), implying
\[
t \, w \cdot (1, \dots, 1) \in \mathbb{Z} \setminus \{0\},
\]
contradicting the assumption \(t \notin \mathbb{Z}\). This contradiction resolves Case~1.

\medskip
\noindent\textbf{Case 2:} Assume now that \((1, \dots, 1) \in H(\pi)\) for a positive measure subset of bad parameters.
We claim that for any open set \(U \subset W_\pi\), for almost every parameter \(\lambda \in W_\pi\) and for any \(h \in H(\pi) \setminus \{0\}\) and \(t \in \mathbb{R}\), either \(t h \in \mathbb{Z}^d\) or
\[
\limsup_{\mathcal{R}^n(\lambda,\pi) \in U \times \{\pi\}}
\|(B^n)^*(\lambda, \pi) \, t h\|_{\mathbb{R}^d/\mathbb{Z}^d} > 0.
\]

For such \(U\), we may choose a cycle \(\gamma_0\) for \(\pi\) with Rauzy--Veech matrix \(B_{\gamma_0}(\pi,\lambda)\) having all entries positive, and construct the generalized simplex \(\Delta_\pi \subseteq U\) as before. We then consider the Rauzy--Veech cocycle over \(\Delta_\pi \times H(\pi)\).

The standard simplex is adapted to this cocycle. By Theorem~\ref{avilaf}, parallel lines to elements of the standard simplex do not intersect the central stable space for a full measure set of parameters \(\lambda \in \Delta_\pi\). By Lemma~\ref{avilac}, the line generated by \(t h\) intersects the weak stable space only in \(H(\pi) \cap \mathbb{Z}^d\) for almost every \(\lambda \in \Delta_\pi\).

This shows that the set of parameters \(\lambda\) for which there exists \(t h \notin \mathbb{Z}^d\) and
\[
\limsup_{\mathcal{R}^n(\lambda,\pi) \in U \times \{\pi\}}
\|(B^n)^*(\lambda, \pi) \, t h\|_{\mathbb{R}^d/\mathbb{Z}^d} = 0
\]
has measure zero.

\medskip
Recall we assumed \(T(\pi, \lambda)\) is ergodic but not weakly mixing. Then there exists \(t \notin \mathbb{Z}\) and \(v = t (1, \dots, 1) \in H(\pi)\) such that there is a solution \(f : \overline{X} \to \mathbb{C}\) to
\[
f(T(x)) = e^{2\pi i v_\alpha} f(x)
\]
for \(x \in X_\alpha \cup X_{\sigma(\alpha)}\).

By the Veech criterion, for generic $\lambda$ such a vector \(v\) must lie in the weak stable space of the Rauzy--Veech cocycle at $\lambda$. The preceding argument shows this is impossible unless \(v \in \mathbb{Z}^d\), contradicting \(t \notin \mathbb{Z}\).

This completes the proof of the theorem.

\end{proof}

\clearpage

\bibliography{bibtex/bib/IEEEexample}

\newcommand{\etalchar}[1]{$^{#1}$}
\begin{thebibliography}{BDG{\etalchar{+}}21}

\bibitem[AF07]{avila2007weak}
Artur Avila and Giovanni Forni.
\newblock Weak mixing for interval exchange transformations and translation flows.
\newblock {\em Annals of mathematics, vol. 165}, pages 637--664, 2007.

\bibitem[AHCF24]{arana2024weak}
Francisco Arana-Herrera, Jon Chaika, and Giovanni Forni.
\newblock Weak mixing in rational billiards.
\newblock {\em arXiv preprint arXiv:2410.11117}, 2024.

\bibitem[AR12]{avila2012exponential}
Artur Avila and Maria~Joao Resende.
\newblock {Exponential mixing for the Teichm{\"u}ller flow in the space of quadratic differentials}.
\newblock {\em Commentarii Mathematici Helvetici}, 87(3):589--638, 2012.

\bibitem[AV07]{avila2007simplicity}
Artur Avila and Marcelo Viana.
\newblock {Simplicity of Lyapunov spectra: proof of the Zorich-Kontsevich conjecture}.
\newblock {\em Acta Math. 198 (1)}, pages 1--56, 2007.

\bibitem[BDG{\etalchar{+}}21]{belldiagonal}
Mark Bell, Vincent Delecroix, Vaibhav Gadre, Rodolfo Guti{\'e}rrez-Romo, and Saul Schleimer.
\newblock Diagonal flow detects the topology of strata.
\newblock {\em arXiv preprint arXiv:2101.12197}, 2021.

\bibitem[BL09]{boissy2009dynamics}
Corentin Boissy and Erwan Lanneau.
\newblock {Dynamics and geometry of the Rauzy--Veech induction for quadratic differentials}.
\newblock {\em Ergodic Theory and Dynamical Systems}, 29(3):767--816, 2009.

\bibitem[DN90]{danthony1990measured}
Claude Danthony and Arnaldo Nogueira.
\newblock Measured foliations on nonorientable surfaces.
\newblock In {\em Annales scientifiques de l'{\'E}cole Normale Sup{\'e}rieure}, pages 469--494, 1990.

\bibitem[FM11]{farb2011primer}
Benson Farb and Dan Margalit.
\newblock {\em A primer on mapping class groups (pms-49)}, volume~41.
\newblock Princeton university press, 2011.

\bibitem[For02]{forni2002deviation}
Giovanni Forni.
\newblock Deviation of ergodic averages for area-preserving flows on surfaces of higher genus.
\newblock {\em Annals of Mathematics}, 155(1):1--103, 2002.

\bibitem[Gad12]{gadre2012dynamics}
Vaibhav~S Gadre.
\newblock Dynamics of non-classical interval exchanges.
\newblock {\em Ergodic Theory and Dynamical Systems}, 32(6):1930--1971, 2012.

\bibitem[GR19]{gutierrez2019classification}
Rodolfo Guti{\'e}rrez-Romo.
\newblock {Classification of Rauzy--Veech groups: proof of the Zorich conjecture}.
\newblock {\em Inventiones mathematicae}, 215:741--778, 2019.

\bibitem[KMS86]{kerckhoff1986ergodicity}
Steven Kerckhoff, Howard Masur, and John Smillie.
\newblock Ergodicity of billiard flows and quadratic differentials.
\newblock {\em Annals of Mathematics}, 124(2):293--311, 1986.

\bibitem[KZ97]{kontsevich1997lyapunov}
Maxim Kontsevich and Anton Zorich.
\newblock {Lyapunov exponents and Hodge theory}.
\newblock {\em arXiv preprint hep-th/9701164}, 1997.

\bibitem[Lan08]{lanneau2008connected}
Erwan Lanneau.
\newblock Connected components of the strata of the moduli spaces of quadratic differentials.
\newblock In {\em Annales scientifiques de l'{\'E}cole normale sup{\'e}rieure}, pages 1--56, 2008.

\bibitem[Mas82]{masur1982interval}
Howard Masur.
\newblock Interval exchange transformations and measured foliations.
\newblock {\em Annals of Mathematics}, 115(1):169--200, 1982.

\bibitem[Mer24]{mercat2024coboundaries}
Paul Mercat.
\newblock Coboundaries and eigenvalues of morphic subshifts.
\newblock {\em arXiv preprint arXiv:2404.13656}, 2024.

\bibitem[Sol24]{solomyak2024note}
Boris Solomyak.
\newblock {A note on spectral properties of random $ S $-adic systems}.
\newblock {\em arXiv preprint arXiv:2403.08884}, 2024.

\bibitem[Tre13]{trevino2013non}
Rodrigo Trevi{\~n}o.
\newblock {On the non-uniform hyperbolicity of the Kontsevich--Zorich cocycle for quadratic differentials}.
\newblock {\em Geometriae Dedicata}, 163(1):311--338, 2013.

\bibitem[Vee84]{veech1984metric}
William~A Veech.
\newblock {The Metric Theory of Interval Exchange Transformations I. Generic Spectral Properties}.
\newblock {\em American Journal of Mathematics}, 106(6):1331--1359, 1984.

\bibitem[Vee86]{veech1986teichmuller}
William~A Veech.
\newblock {The Teichm{\"u}ller GeodesicFlow}.
\newblock {\em Annals of Mathematics}, 124(3):441--530, 1986.

\bibitem[Via06]{viana2006ergodic}
Marcelo Viana.
\newblock Ergodic theory of interval exchange maps.
\newblock {\em Revista Matem{\'a}tica Complutense}, 19(1):7--100, 2006.

\end{thebibliography}
\end{document}